\RequirePackage{fix-cm}
\documentclass[smallextended]{svjour3}       
\smartqed  
\usepackage{graphicx}

\usepackage{euscript,amsmath,amssymb,amsfonts,graphicx,bm,amsthm}
\usepackage{latexsym,amssymb,enumerate,amsmath,amsthm} 
\usepackage{soul,xcolor}

\newcommand{\dd}{\text{d}}
\def\eps{\varepsilon}
\def\E{\mathbb{E}}
\def\P{\mathbb{P}}
\def\R{\mathbb{R}}
\def\tod{\to_{\textup{d}}}
\def\T{\mathcal{T}}

\def\black{\color{black}}

\theoremstyle{plain}

\usepackage{natbib}
\setcitestyle{authoryear}

\begin{document}

\title{Distribution of extreme first passage times of diffusion\thanks{The author was supported by the National Science Foundation (Grant Nos.\ DMS-1944574, DMS-1814832, and DMS-1148230).}
}


\author{Sean D. Lawley
}


\institute{Sean D. Lawley \at
              University of Utah, Salt Lake City, UT 84112 USA \\
              \email{lawley@math.utah.edu}           
}

\date{Received: date / Accepted: date}

\maketitle

\begin{abstract}
Many events in biology are triggered when a diffusing searcher finds a target, which is called a first passage time (FPT). The overwhelming majority of FPT studies have analyzed the time it takes a single searcher to find a target. However, the more relevant timescale in many biological systems is the time it takes the fastest searcher(s) out of many searchers to find a target, which is called an extreme FPT. In this paper, we apply extreme value theory to find a tractable approximation for the full probability distribution of extreme FPTs of diffusion. This approximation can be easily applied in many diverse scenarios, as it depends on only a few properties of the short time behavior of the survival probability of a single FPT. We find this distribution by proving that a careful rescaling of extreme FPTs converges in distribution as the number of searchers grows. This limiting distribution is a type of Gumbel distribution and involves the LambertW function. This analysis yields new explicit formulas for approximations of statistics of extreme FPTs (mean, variance, moments, etc.)\ which are highly accurate and are accompanied by rigorous error estimates.
\end{abstract}

\newpage
\section{Introduction} 

Events in biological systems are often triggered when a diffusing searcher finds a target \citep{chou_first_2014,holcman_time_2014,bressloff_stochastic_2013}. Examples range from the initiation of the immune response when a searching T cell finds a cognate antigen \citep{delgado2015}, to the triggering of calcium release by diffusing IP$_{3}$ molecules that reach IP$_{3}$  receptors \citep{wang1995}, to gene activation by the arrival of a diffusing transcription factor to a certain gene \citep{larson2011}, to animals foraging for food \citep{mckenzie2009,kurella2015}. In such systems, the activation timescale is determined by the \emph{first passage time} (FPT) of a searcher to a target.

The vast majority of FPT studies have focused on the time it takes a given single searcher to find a target. However, several recent works and commentaries have shown that the relevant timescale in many biological systems is actually the time it takes the fastest searcher(s) to find a target out of a large group of searchers \citep{schuss2019,basnayake2019,coombs2019,redner2019,sokolov2019,rusakov2019,martyushev2019,tamm2019,basnayake_extreme_2018,guerrier2018}. For example, approximately $N=10^{8}$ sperm cells search for an egg in human reproduction, but fertilization occurs as soon as a single sperm cell finds the egg \citep{meerson2015,reynaud2015,barlow2016,yang2016}.

Importantly, the time it takes the fastest searcher(s) out of many searchers to find a target is typically much less than the time it takes a given single searcher to find a target. In fact, \cite{schuss2019} postulated that this is a general mechanism that operates across many biological systems and called it the \emph{redundancy principle}. In particular, these authors claimed that many seemingly redundant copies of a searcher (molecule, protein, cell, animal, etc.)\ are not superfluous, but rather have the specific functions of accelerating activation rates. That is, the apparently ``extra'' copies are in fact necessary for biological function.

Indeed, the review by \cite{schuss2019} highlights many examples of signal transduction triggered by the fastest molecules to find a target. We now explain one example recently studied by \cite{basnayake2019fast} that involves calcium-induced calcium release in dendritic spines. While the geometry can vary greatly, a dendritic spine consists roughly of a bulbous head connected to a thin neck. It has been observed that calcium ions entering at the head of the spine can diffuse to and then bind small Ryanodyne receptors at the base of the spine neck which then induces an avalanche of calcium release from internal spine apparatus stores. This calcium avalanche at the base occurs only a few milliseconds after calcium ions enter at the head, which is perplexing because the time it takes a given single calcium ion to diffuse from the head to a Ryanodyne receptor at the base is approximately $\tau=120$ milliseconds. However, through a close integration of experiments and numerical simulations, \cite{basnayake2019fast} explained this phenomenon by showing that approximately $N=10^{3}$ ions enter at the head and that the fastest ions out of this group take only a few milliseconds to reach the receptors at the base. Similar processes occur in the photoresponse of a fly to the absorption of a single photon \citep{katz2017,schuss2019}.

Another similar example concerns the random production of antibodies by genetic recombination inside a B cell during somatic hypermutation \citep{schuss2019,coombs2019}. In this scenario, while several gene segment copies are produced, only the first segment to find and bind a certain macromolecular complex will be used for producing antibodies. Additional examples include the IP3 pathway, in which the first IP3 molecules which find small IP3 receptors induce calcium release, and synaptic transmission, in which the pre-synaptic signal is transmitted by the first of many neurotransmitters which diffuse to and bind small post-synaptic receptors \citep{schuss2019}.

To investigate how the number of searchers affects the time it takes the fastest searcher(s) to find a target, consider $N\gg1$ independent and identical diffusive searchers. Let $\tau_{1},\dots,\tau_{N}$ be their independent and identically distributed (iid) FPTs to reach some target. The first time one of these searchers finds the target is
\begin{align}\label{ffpt}
T_{N}
:=\min\{\tau_{1},\dots,\tau_{N}\}.
\end{align}
More generally, the $k$th fastest searcher finds the target at time
\begin{align}\label{tkn}
T_{k,N}
:=\min\big\{\{\tau_{1},\dots,\tau_{N}\}\backslash\cup_{j=1}^{k-1}\{T_{j,N}\}\big\},\quad k\in\{1,\dots,N\},
\end{align}
where $T_{1,N}:=T_{N}$.

{\black The mean of a single FPT, $\E[\tau_{1}]$, is well understood in a variety of scenarios \citep{benichou2010,cheviakov2010,PB3,lindsay2017,lawley2019dtmfpt}, and important progress has been made recently in understanding the distribution of a single FPT \citep{rupprecht2015,grebenkov2018strong,grebenkov2018,grebenkov2019full}. However, studying the } so-called \emph{extreme} FPTs, $T_{k,N}$, is notoriously difficult, both analytically and numerically \citep{weiss1983,yuste1996,basnayake2019,schuss2019,lawley2020esp,lawley2020uni}. An essential difficulty is that extreme FPTs depend on very rare events. Indeed, while a typical searcher tends to wander around before finding the target, the fastest searchers move almost deterministically along the shortest geodesic path to the target  {\black \citep{godec2016, godec2016x, grebenkov2017mortal, basnayake_extreme_2018, lawley2020uni}}. This phenomenon is illustrated in Figure~\ref{figschem}.

\begin{figure}
\centering
\includegraphics[width=.4\linewidth]{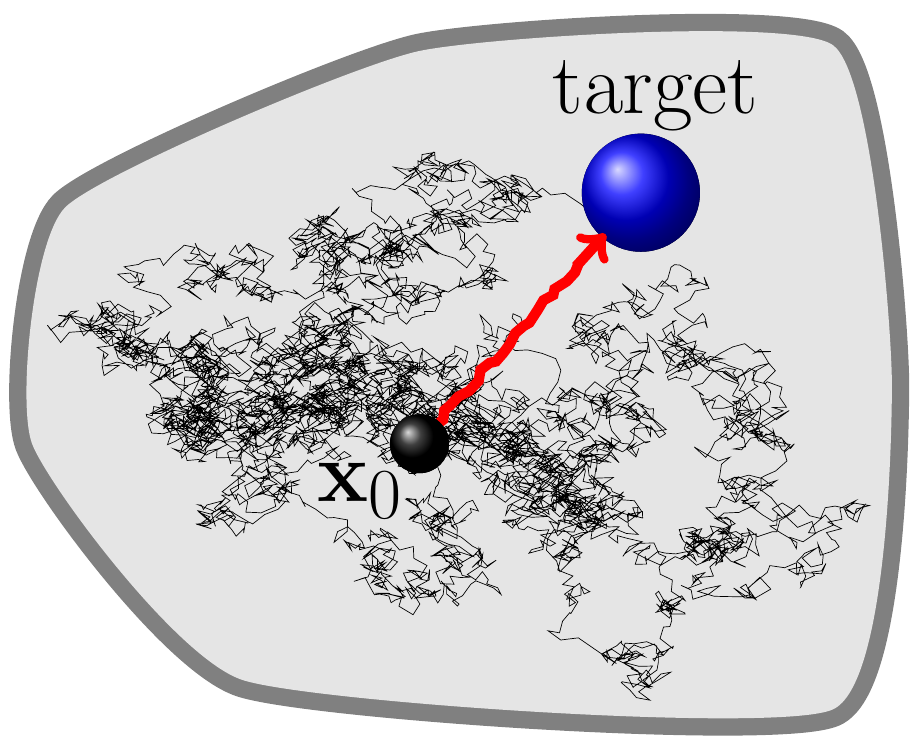}
\caption{The fastest diffusive searcher out of $N\gg1$ searchers moves almost deterministically (red trajectory) along the shortest path from the starting location $\mathbf{x}_{0}$ to the target (blue ball), while a typical searcher wanders around (black trajectory) before finding the target.}
\label{figschem}
\end{figure}

Another significant challenge in understanding extreme FPTs in biological applications is that the targets are often very small \citep{holcman2014}. For example, this is the case in the application to calcium-induced calcium release in dendritic spines discussed above, as a Ryanodyne receptor can be modeled as a disk of radius $r=0.01\,\mu\text{m}$ whereas the distance from the spine head to the base of the spine neck is roughly $L=3\,\mu\text{m}$ \citep{basnayake2019fast}, and thus a dimensionless measure of the target size is
\begin{align*}
\eps
:=r/L
\approx0.003\ll1.
\end{align*}
This is a challenge because the mean fastest FPT, $\E[T_{N}]$, diverges for small targets ($\eps\ll1$) but vanishes for many searchers ($N\gg1$). That is, if we fix the number of searchers $N$ and take the target size $\eps$ sufficiently small, then
\begin{align}\label{smalltargets}
\tfrac{D}{L^{2}}\E[T_{N}]\gg1,
\end{align}
where $\tfrac{L^{2}}{D}$ is the diffusion timescale. On the other hand, if we fix the target size $\eps$ and take the number of searchers $N$ sufficiently large, then
\begin{align}\label{largeN}
\tfrac{D}{L^{2}}\E[T_{N}]\ll1.
\end{align}

Hence, as a first step in any specific biological application involving extreme FPTs with small targets ($\eps\ll1$) and many searchers ($N\gg1$), one needs to determine if the extreme FPTs are in the regime represented by either \eqref{smalltargets} or \eqref{largeN}. For example, in the dendritic spine application described above, it is not \emph{a priori} clear that $N=10^{3}$ is sufficiently large to overcome the small Ryanodyne receptors ($\eps\approx0.003$) and make the extreme FPT on the order of only 2-3 milliseconds (which is much less than the diffusion timescale in this problem, $\tfrac{L^{2}}{D}\approx15$ milliseconds). Indeed, \cite{basnayake2019fast} developed detailed numerical Monte Carlo simulations to reach this conclusion.

Importantly, analytical approximations of statistics of extreme FPTs for small targets in general 3-dimensional domains are lacking. \cite{basnayake2019} derived a formal approximation, but this was proven to be false \citep{lawley2020esp}. Recent work found the leading order large $N$ behavior of all the moments of $T_{k,N}$, but it turns out this leading order behavior is independent of the target size \citep{lawley2020uni}. Hence, these results cannot determine if a particular application is in the regime represented by \eqref{smalltargets} or \eqref{largeN}.

In this paper, we apply the theory of extreme statistics to find a tractable approximation for the full probability distribution of extreme FPTs of diffusion. This rigorous approximation can be applied in many scenarios as it depends on only a few properties of the short time behavior of the survival probability of a single FPT. Indeed, as long as this short time behavior is known, this approximation can be immediately applied to scenarios involving small targets and thus can determine the influence of the competing limits of small targets ($\eps\ll1$) and many searchers ($N\gg1$).

We find this distribution by proving that a careful rescaling of extreme FPTs converges in distribution as the number of searchers grows. This limiting distribution is a type of Gumbel distribution and involves the so-called LambertW function (defined as the inverse of $f(z)=ze^{z}$ \citep{corless1996}). This analysis yields new explicit formulas for statistics of extreme FPTs (mean, variance, moments, etc.). These formulas are highly accurate and are accompanied by rigorous error estimates. Further, these formulas confirm and explain a conjecture by \cite{yuste2001} that extreme FPT statistics can be approximated by a certain infinite series involving iterated logarithms.

The rest of the paper is organized as follows. We first summarize our main results in section~\ref{main results}. In section~\ref{math}, we develop and state our precise mathematical results in more detail. We then illustrate these general results in a few examples in section~\ref{examples}. In the Discussion section, we describe relations to prior work and discuss applications of the theory. Finally, we collect all the mathematical proofs in an appendix.

\section{Main results}\label{main results}

Let $\{\tau_{n}\}_{n\ge1}$ be an iid sequence of FPTs with survival probability
\begin{align*}
S(t):=\P(\tau_{1}>t).
\end{align*}
Assume that $S(t)$ has the short time behavior,
\begin{align}\label{st}
1-S(t)
\sim At^{p}e^{-C/t}\quad\text{as }t\to0+,
\end{align}
for some constants $A>0$, $C>0$, and $p\in\R$. Throughout this work,
\begin{align*}
\text{``$f\sim g$'' means }f/g\to1.
\end{align*}
We emphasize that \eqref{st} is a generic behavior for diffusion processes that holds in many diverse scenarios (see the Discussion section for more details).

Letting $T_{N}$ denote the fastest FPT in \eqref{ffpt}, we prove the following convergence in distribution (Theorem~\ref{main}),
\begin{align}\label{cdmr}
\frac{T_{N}-b_{N}}{a_{N}}
\to_{\textup{d}} X\quad\text{as }N\to\infty,
\end{align}
where $X$ has a standard Gumbel distribution, $\P(X>x)=\exp(-e^{x})$, and
\begin{align}\label{ababmr}
\begin{split}
a_{N}
&
=\frac{b_{N}}{\ln(AN)},\quad
b_{N}
=\frac{C}{\ln(AN)},\quad\text{if }p=0,\\
a_{N}
&
=\frac{b_{N}}{p(1+W)},\quad
b_{N}
=\frac{C}{pW},\quad\text{if }p\neq0,
\end{split}
\end{align}
and
\begin{align*}
W
&=
\begin{cases}
W_{0}\big((C/p)(AN)^{1/p}\big) & \text{if }p>0,\\
W_{-1}\big((C/p)(AN)^{1/p}\big) & \text{if }p<0,
\end{cases}
\end{align*}
where $W_{0}(z)$ denotes the principal branch of the LambertW function and $W_{-1}(z)$ denotes the lower branch \citep{corless1996}. The LambertW function is a fairly standard function that is included in most modern computational software (it is sometimes called the product logarithm or the omega function). {\black Theorem~\ref{uf} gives the following alternative formulas for $a_{N}$ and $b_{N}$ which avoid the LambertW function,
\begin{align}\label{primes}
a_{N}'
&=\frac{C}{(\ln N)^{2}},\quad
b_{N}'
=\frac{C}{\ln N}
+\frac{Cp\ln(\ln(N))}{(\ln N)^{2}}
-\frac{C\ln(AC^{p})}{(\ln N)^{2}}.
\end{align}
In particular, all the statements in this section hold with $a_{N},b_{N}$ replaced by $a_{N}',b_{N}'$.}

The convergence in distribution in \eqref{cdmr} means that if $N\gg1$, then the distribution of the fastest FPT, $T_{N}$, is approximately Gumbel with shape parameter $b_{N}$ and scale parameter $a_{N}$. That is,
\begin{align*}
\P(T_{N}>t)
\approx\exp\Big[-\exp\Big(\frac{t-b_{N}}{a_{N}}\Big)\Big]\quad\text{if }N\gg1,
\end{align*}
where $a_{N},b_{N}$ are in \eqref{ababmr} {\black (or are replaced by $a_{N}',b_{N}'$ in \eqref{primes})}. Note that essentially all the statistical information about a Gumbel distribution is immediately available (mean, median, mode, variance, moments, probability density function, etc., see Proposition~\ref{basic} below). Therefore, this result provides all the statistical information for the fastest FPT (approximately for large $N$). For example, we prove that if $\E[T_{N}]<\infty$ for some $N\ge1$, then (Theorem~\ref{moments})
\begin{align*}
\E[T_{N}]
&=b_{N}-\gamma a_{N}+o(a_{N}),\\
\textup{Variance}(T_{N})
&=\frac{\pi^{2}}{6}a_{N}^{2}+o(a_{N}^{2}),
\end{align*}
where $\gamma\approx0.5772$ is the Euler-Mascheroni constant and $f(N)=o(a_{N}^{m})$ means $\lim_{N\to\infty}a_{N}^{-m}f(N)=0$.

We prove similar results for the $k$th fastest FPT, $T_{k,N}$, defined in \eqref{tkn}. In particular, we prove that the joint distribution of a rescaling of the $k$ fastest FPTs,
\begin{align*}
\left(\frac{T_{1,N}-b_{N}}{a_{N}},\dots,\frac{T_{k,N}-b_{N}}{a_{N}}\right),
\end{align*}
converges as $N\to\infty$ to a distribution that we give explicitly (Theorem~\ref{kth}). This result provides explicit approximations for statistics of $T_{k,N}$, including (Theorem~\ref{kth moment}),
\begin{align*}
\E[T_{k,N}]
&=b_{N}+\psi(k)a_{N}+o(a_{N})
=\E[T_{N}]+H_{k-1}a_{N}+o(a_{N}),\\
\textup{Variance}(T_{k,N})
&=\psi'(k)a_{N}^{2}+o(a_{N}^{2}),
\end{align*}
where $\psi(x)$ is the digamma function and $H_{k-1}=\sum_{r=1}^{k-1}\tfrac{1}{r}$ is the $(k-1)$-th harmonic number.

\section{Mathematical analysis}\label{math}

\subsection{Fastest FPT}

Let $\{\tau_{n}\}_{n\ge1}$ be an iid sequence of FPTs with survival probability $S(t):=\P(\tau_{1}>t)$. Define the fastest FPT, $T_{N}$, as in \eqref{ffpt}. Since the sequence $\{\tau_{n}\}_{n\ge1}$ is iid, it is immediate that the survival probability of $T_{N}$ is
\begin{align}\label{exact}
\P(T_{N}>t)
=(\P(\tau_{1}>t))^{N}
=(S(t))^{N}.
\end{align}
While \eqref{exact} is the exact distribution of $T_{N}$, this formula it is not particularly useful for understanding how the distribution depends on parameters or for calculating statistics of $T_{N}$. Furthermore, the full survival probability $S(t)$ of a single FPT is often unknown.

We thus seek a tractable approximation of \eqref{exact} for large $N$, which will thus depend only on the short time behavior of $S(t)$. Now, \eqref{exact} implies that the limiting distribution of $T_{N}$ for large $N$ is trivial,
\begin{align*}
\lim_{N\to\infty}\P(T_{N}>t)
=\begin{cases}
1 & \text{if }t<t^{*},\\
0 & \text{if }t>t^{*},
\end{cases}
\end{align*}
where $t^{*}:=\inf\{t>0:S(t)<1\}$. For nontrivial diffusion processes, we typically have $t^{*}=0$. To ameliorate this problem, we study the distribution of $T_{N}$ by finding a rescaling of $T_{N}$ that has a nontrivial limiting distribution for large $N$. Specifically, we find sequences $\{a_{N}\}_{N\ge1}$ and $\{b_{N}\}_{N\ge1}$ so that
\begin{align*}
\frac{T_{N}-b_{N}}{a_{N}}
\to_{\textup{d}}X\quad\text{as }N\to\infty,
\end{align*}
for some random variable $X$. In this paper, $\to_{\textup{d}}$ denotes convergence in distribution \citep{billingsley2013}, which means
\begin{align}\label{conv}
\P\left(\frac{T_{N}-b_{N}}{a_{N}}>x\right)
=(S(a_{N}x+b_{N}))^{N}
\to G(x)\quad\text{as }N\to\infty,
\end{align}
{\black for all $x\in\R$ where $G(x)=\P(X>x)$ is continuous.}

Remarkably, the Fisher-Tippett-Gnedenko Theorem states that if \eqref{conv} holds for a nondegenerate $G$, then $G$ must be either a Weibull, Frechet, or Gumbel distribution \citep{fisher1928}. This theorem is the cornerstone of extreme value theory, and applies to the minimum or maximum of any sequence of iid random variables \citep{colesbook,haanbook,falkbook}. Since the limiting distribution must be one of these three types, this classical theorem is an extreme value analog of the central limit theorem. We prove below that the typical short time behavior of $S(t)$ ensures that $G$ must be Gumbel. The following definition and proposition collects some facts about the Gumbel distribution.

\begin{definition}
A random variable $X$ has a \emph{Gumbel distribution} with location parameter $b\in\R$ and scale parameter $a>0$ if\footnote{Some authors define a Gumbel distribution slightly differently, by saying that $-X$ has a Gumbel distribution with shape $-b$ and scale $a$ if \eqref{xgumbel} holds.}
\begin{align}\label{xgumbel}
\P(X>x)
=\exp\Big[-\exp\Big(\frac{x-b}{a}\Big)\Big],\quad\text{for all } x\in\R.
\end{align}
If \eqref{xgumbel} holds, then we write
\begin{align*}
X=_{\textup{d}}\textup{Gumbel}(b,a).
\end{align*}
\end{definition}

\begin{proposition}\label{basic}
If $X=_{\textup{d}}\textup{Gumbel}(b,a)$, then its survival probability is in \eqref{xgumbel}, its probability density function is
\begin{align*}
f_{X}(x)
=\frac{1}{a}\exp\Big[\frac{x-b}{a}-\exp\Big(\frac{x-b}{a}\Big)\Big],\quad x\in\R,
\end{align*}
and its moment generating function is
\begin{align*}
M_{X}(t)
:=\E[e^{tX}]
=\Gamma(1+at)e^{bt},\quad t\in\R,
\end{align*}
where $\Gamma(\cdot)$ denotes the gamma function. Hence, the mean and variance are
\begin{align*}
\E[X]
=b-\gamma a,
\quad
\textup{Variance}(X)
=\frac{\pi^{2}}{6}a^{2},
\end{align*}
where $\gamma\approx0.5772$ is the Euler-Mascheroni constant. The mode and median are
\begin{align*}
\textup{Mode}(X)
=b,
\quad
\textup{Median}(X)
=b+a\ln(\ln(2))
\approx b-0.3665a.
\end{align*}
\end{proposition}

Now, it was recently proven that under very general assumptions, the survival probability of a single diffusive FPT has the following short time behavior,
\begin{align}\label{C}
\lim_{t\to0+}t\ln(1-S(t))=-C<0,
\end{align}
where $C=L^{2}/(4D)>0$ and $D$ is a characteristic diffusivity and $L$ is a certain geodesic distance \citep{lawley2020uni}, as long as the diffusive searchers cannot start arbitrarily close to the target. The next proposition shows that if \eqref{C} holds, then any nondegenerate limiting distribution $G$ in \eqref{conv} must be Gumbel.

\begin{proposition}\label{easy}
Let $\{\tau_{n}\}_{n\ge1}$ be an iid sequence of nonnegative random variables with $S(t):=\P(\tau_{1}>t)$, define $T_{N}:=\min\{\tau_{1},\dots,\tau_{N}\}$, and suppose \eqref{C} holds. If there exists sequences $\{a_{N}\}_{N\ge1}$ and $\{b_{N}\}_{N\ge1}$ with $a_{N}>0$ and $b_{N}\in\R$ so that
\begin{align*}
\frac{T_{N}-b_{N}}{a_{N}}
\to_{\textup{d}}X\quad\text{as }N\to\infty,
\end{align*}
and $X$ has a nondegenerate distribution, then $X=_{\textup{d}}\textup{Gumbel}(b,a)$ for some $b\in\R$ and $a>0$.
\end{proposition}

The condition in \eqref{C} implies that
\begin{align*}
S(t)
= 1-e^{-C/t+h(t)},
\end{align*}
where $h(t)$ is some function satisfying $th(t)\to0$ as $t\to0+$. The following proposition gives precise conditions on $h(t)$ which yield rescalings $\{a_{N}\}_{N\ge1}$ and $\{b_{N}\}_{N\ge1}$ so that $(T_{N}-b_{N})/a_{N}$ converges in distribution to a Gumbel random variable.

\begin{proposition}\label{hh}
Let $\{\tau_{n}\}_{n\ge1}$ be an iid sequence of nonnegative random variables with $S(t):=\P(\tau_{1}>t)$, define $T_{N}:=\min\{\tau_{1},\dots,\tau_{N}\}$, and assume
\begin{align*}
1-S(t)
\sim1-S_{0}(t)\quad\text{as }t\to0+,
\end{align*}
where
\begin{align*}
S_{0}(t)
=
1-e^{-C/t+h(t)},\quad\text{if }t>0,
\end{align*}
for some constant $C>0$ and some function $h(t)$ that is twice-continuously differentiable for $t>0$ and satisfies
\begin{align}\label{hc}
\lim_{t\to0+}th(t)
=\lim_{t\to0+}t^{2}h'(t)
=\lim_{t\to0+}t^{4}h''(t)
=0.
\end{align}
Then
\begin{align*}
\frac{T_{N}-b_{N}}{a_{N}}
\to_{\textup{d}}
X=_{\textup{d}}\textup{Gumbel}(0,1)\quad\text{as }N\to\infty,
\end{align*}
where
\begin{align}\label{ab}
a_{N}
&:=\frac{-1}{NS_{0}'(b_{N})}>0,\quad
b_{N}
:=S_{0}^{-1}(1-1/N)>0,\quad N\ge1.
\end{align}
\end{proposition}

{\black In \eqref{ab}, $S_{0}^{-1}(1-1/N)$ denotes the inverse of $S_{0}$, which must exist for large $N$ by the assumptions on $h(t)$.} As we will see, it is typically the case that $h(t)=\ln(At^{p})$, which clearly satisfies \eqref{hc}. In this case, we work out the rescalings $\{a_{N}\}_{N\ge1}$ and $\{b_{N}\}_{N\ge1}$.

\begin{theorem}\label{main}
Let $\{\tau_{n}\}_{n\ge1}$ be an iid sequence of nonnegative random variables with $S(t):=\P(\tau_{1}>t)$, define $T_{N}:=\min\{\tau_{1},\dots,\tau_{N}\}$, and assume there exists constants $C>0$, $A>0$, and $p\in\R$ so that
\begin{align*}
1-S(t)\sim At^{p}e^{-C/t}\quad\text{as }t\to0+.
\end{align*}
Then
\begin{align}\label{cd}
\frac{T_{N}-b_{N}}{a_{N}}
\to_{\textup{d}}
X=_{\textup{d}}\textup{Gumbel}(0,1)\quad\text{as }N\to\infty,
\end{align}
where 
\begin{align}\label{abab}
\begin{split}
a_{N}
&
=\frac{b_{N}}{\ln(AN)},\quad
b_{N}
=\frac{C}{\ln(AN)},\quad\text{if }p=0,\\
a_{N}
&
=\frac{b_{N}}{p(1+W)},\quad
b_{N}
=\frac{C}{pW},\quad\text{if }p\neq0,
\end{split}
\end{align}
and
\begin{align}\label{dubdub}
W
&=
\begin{cases}
W_{0}\big((C/p)(AN)^{1/p}\big) & \text{if }p>0,\\
W_{-1}\big((C/p)(AN)^{1/p}\big) & \text{if }p<0,
\end{cases}
\end{align}
where $W_{0}(z)$ denotes the principal branch of the LambertW function and $W_{-1}(z)$ denotes the lower branch \citep{corless1996}.
\end{theorem}

If the convergence in distribution in \eqref{cd} holds for some rescalings $\{a_{N}\}_{N\ge1}$ and $\{b_{N}\}_{N\ge1}$, then we also have that
\begin{align}\label{cc0}
\frac{T_{N}-b_{N}'}{a_{N}'}
\to_{\textup{d}}
X=_{\textup{d}}\textup{Gumbel}(0,1)\quad\text{as }N\to\infty,
\end{align}
for any rescalings $\{a_{N}'\}_{N\ge1}$ and $\{b_{N}'\}_{N\ge1}$ that satisfy \citep{peng2012}
\begin{align}\label{cc1}
\lim_{N\to\infty}\frac{a_{N}'}{a_{N}}
=1,\quad
\lim_{N\to\infty}\frac{b_{N}'-b_{N}}{a_{N}}
=0.
\end{align}
The following theorem gives rescalings which avoid the LambertW functions used in Theorem~\ref{main} and are valid for any $p\in\R$.

\begin{theorem}\label{uf}
Under the assumptions of Theorem~\ref{main}, we have that
\begin{align*}
\frac{T_{N}-b_{N}'}{a_{N}'}
\to_{\textup{d}}
X=_{\textup{d}}\textup{Gumbel}(0,1)\quad\text{as }N\to\infty,
\end{align*}
where 
\begin{align}\label{ababuf}
\begin{split}
a_{N}'
&=\frac{C}{(\ln N)^{2}},\quad
b_{N}'
=\frac{C}{\ln N}
+\frac{Cp\ln(\ln(N))}{(\ln N)^{2}}
-\frac{C\ln(AC^{p})}{(\ln N)^{2}}.
\end{split}
\end{align}
\end{theorem}

The conclusions of Propositions~\ref{easy}-\ref{hh} and Theorems~\ref{main}-\ref{uf} concern convergence in distribution. In general, convergence in distribution does not imply moment convergence \citep{billingsley2013}. That is, $X_{N}\tod X$ does not necessarily imply $\E[(X_{N})^{m}]\to\E[X^{m}]$ for $m>0$. However, \cite{pickands1968} proved that convergence in distribution does imply moment convergence for extreme values.

\begin{theorem}\label{moments}
Under the assumptions of Theorem~\ref{main} with $\{a_{N}\}_{N\ge1}$ and $\{b_{N}\}_{N\ge1}$ given by either \eqref{abab} or \eqref{ababuf}, assume further that $\E[T_{N}]<\infty$ for some $N\ge1$. Then for each moment $m\in(0,\infty)$, we have that
\begin{align*}
\E\left[\left(\frac{T_{N}-b_{N}}{a_{N}}\right)^{m}\right]
\to\E[X^{m}]\quad\text{as }N\to\infty,\quad\text{where }X=_{\textup{d}}\textup{Gumbel}(0,1).
\end{align*}
Therefore,
\begin{align*}
\E[(T_{N}-b_{N})^{m}]
=a_{N}^{m}\E[X^{m}]+o(a_{N}^{m}),
\end{align*}
where $f(N)=o(a_{N}^{m})$ means $\lim_{N\to\infty}a_{N}^{-m}f(N)=0$. Further, if $m>0$ is an integer, then $\E[X^{m}]$ can be calculated explicitly by Proposition~\ref{basic}. For example, we have that
\begin{align*}
\E[T_{N}]
&=b_{N}-\gamma a_{N}+o(a_{N}),\\
\textup{Variance}(T_{N})
&=\frac{\pi^{2}}{6}a_{N}^{2}+o(a_{N}^{2}),
\end{align*}
where $\gamma\approx0.5772$ is the Euler-Mascheroni constant.
\end{theorem}

\subsection{$k$th fastest FPT}

We now extend the results in the previous subsection on the fastest FPT to the $k$th fastest FPT,
\begin{align*}
T_{k,N}
:=\min\big\{\{\tau_{1},\dots,\tau_{N}\}\backslash\cup_{j=1}^{k-1}\{T_{j,N}\}\big\},\quad k\in\{1,\dots,N\},
\end{align*}
where $T_{1,N}:=T_{N}$.

\begin{theorem}\label{kth}
Under the assumptions of Theorem~\ref{main} with $\{a_{N}\}_{N\ge1}$ and $\{b_{N}\}_{N\ge1}$ given by either \eqref{abab} or \eqref{ababuf}, we have that for each fixed $k\ge1$, 
\begin{align}\label{cd1}
\frac{T_{k,N}-b_{N}}{a_{N}}
\to_{\textup{d}}
X_{k}\quad\text{as }N\to\infty,
\end{align}
where $X_{k}$ has the probability density function,
\begin{align}\label{xkpdf}
f_{X_{k}}(x)
=\frac{\exp(kx-e^{x})}{(k-1)!},\quad\text{for all } x\in\R.
\end{align}
Furthermore, for each fixed $k\ge1$, we have the following convergence in distribution for the joint random variables,
\begin{align}\label{cd2}
\left(\frac{T_{1,N}-b_{N}}{a_{N}},\dots,\frac{T_{k,N}-b_{N}}{a_{N}}\right)
\to_{\textup{d}}
\mathbf{X}^{(k)}
=(X_{1},\dots,X_{k})\in\R^{k}\quad\text{as }N\to\infty,
\end{align}
where the joint probability density function of $\mathbf{X}^{(k)}\in\R^{k}$ is
\begin{align*}
f_{\mathbf{X}^{(k)}}(x_{1},\dots,x_{k})
&=\begin{cases}
\exp(-e^{x_{k}})\prod_{r=1}^{k}e^{x_{r}} & \text{if $x_{1}\le\dots\le x_{k}$},\\
0 & \text{otherwise}.
\end{cases}
\end{align*}
\end{theorem}

The following theorem ensures the convergence of the moments of the $k$th fastest FPT.

\begin{theorem}\label{kth moment}
Under the assumptions of Theorem~\ref{moments} with $\{a_{N}\}_{N\ge1}$ and $\{b_{N}\}_{N\ge1}$ given by either \eqref{abab} or \eqref{ababuf}, we have that for each moment $m\in(0,\infty)$,
\begin{align}\label{kmoment}
\E\left[\left(\frac{T_{k,N}-b_{N}}{a_{N}}\right)^{m}\right]
\to\E[X_{k}^{m}]\quad\text{as }N\to\infty,
\end{align}
where $X_{k}$ has the probability density function in \eqref{xkpdf}. Therefore,
\begin{align*}
\E[(T_{k,N}-b_{N})^{m}]
=a_{N}^{m}\E[X_{k}^{m}]+o(a_{N}^{m}).
\end{align*}
Further, if $m>0$ is an integer, then $\E[X_{k}^{m}]$ can be explicitly calculated. In particular,
\begin{align*}
\E[T_{k,N}]
&=b_{N}+\psi(k)a_{N}+o(a_{N})
=\E[T_{1,N}]+H_{k-1}a_{N}+o(a_{N}),\\
\textup{Variance}(T_{k,N})
&=\psi'(k)a_{N}^{2}+o(a_{N}^{2}),
\end{align*}
where $\psi(x)$ is the digamma function and $H_{k-1}=\sum_{r=1}^{k-1}\tfrac{1}{r}$ is the $(k-1)$-th harmonic number.
\end{theorem}

\section{Numerical examples}\label{examples}

We now apply our results to three specific examples.

\subsection{One dimension}\label{1d}

First consider the case of $N\ge1$ independent searchers diffusing in one space dimension with diffusivity $D>0$. Suppose the searchers each start at ${L}>0$ and let $\tau_{n}$ be the first time the $n$th searcher reaches the origin. In this case,
\begin{align*}
S(t)=\P(\tau_{1}>t)=1-\text{erfc}\Big(\frac{{{L}}}{\sqrt{4D t}}\Big),\quad t>0,
\end{align*}
and thus
\begin{align*}
1-S(t)
\sim \sqrt{\frac{4Dt}{\pi {L}^{2}}}e^{-{L}^{2}/(4Dt)}\quad\text{as }t\to0+.
\end{align*}

Therefore, Theorems~\ref{main}-\ref{kth moment} hold with
\begin{align*}
A=\sqrt{\frac{4D}{\pi {L}^{2}}},\quad
p=\frac{1}{2},\quad
C=\frac{{L}^{2}}{4D}.
\end{align*}
In particular, 
\begin{align*}
\frac{T_{N}-b_{N}}{a_{N}}
\to_{\textup{d}}
X=_{\textup{d}}\textup{Gumbel}(0,1)\quad\text{as }N\to\infty,
\end{align*}
where $\{a_{N}\}_{N\ge1}$ and $\{b_{N}\}_{N\ge1}$ are given by either \eqref{abab} or \eqref{ababuf}. Hence, the distribution of $T_{N}$ is approximately $\textup{Gumbel}(b_{N},a_{N})$.

In the left panel of Figure~\ref{figmean}, we plot the error of various approximations of the mean fastest FPT, $\E[T_{N}]$, as functions of $N$. Specifically, we plot the relative error,
\begin{align}\label{re}
\Big|\frac{\E[T_{N}]-\T_{N}}{\E[T_{N}]}\Big|,
\end{align}
where $\T_{N}$ is an approximation of $\E[T_{N}]$. The value of $\E[T_{N}]$ used in \eqref{re} is calculated by numerical approximation of the following integral,
\begin{align*}
\E[T_{N}]
=\int_{0}^{\infty}(S(t))^{N}\,\dd t.
\end{align*}

The red dotted curve in the left panel of Figure~\ref{figmean} is the error \eqref{re} for the approximation $\T_{N}=L^{2}/(4D\ln N)$ (this approximation dates back to \cite{weiss1983}). The blue dashed curve is for the approximation $\T_{N}=b_{N}'-\gamma a_{N}'$ where $\{a_{N}'\}_{N\ge1}$ and $\{b_{N}'\}_{N\ge1}$ are given by \eqref{ababuf}. The black solid curve is for the approximation $\T_{N}=b_{N}-\gamma a_{N}$ where $\{a_{N}\}_{N\ge1}$ and $\{b_{N}\}_{N\ge1}$ are given by \eqref{abab}. This figure shows that our approximations, $b_{N}-\gamma a_{N}$ and $b_{N}'-\gamma a_{N}'$, to the mean fastest FPT are much more accurate than $L^{2}/(4D\ln N)$, and $b_{N}-\gamma a_{N}$ is more accurate than $b_{N}'-\gamma a_{N}'$.

In the left panel of Figure~\ref{figdist}, we illustrate the convergence in distribution of $(T_{N}-b_{N})/a_{N}$ to $X=_{\textup{d}}\textup{Gumbel}(0,1)$ where $\{a_{N}\}_{N\ge1}$ and $\{b_{N}\}_{N\ge1}$ are given by \eqref{abab}. Specifically, we plot the probability density function of $(T_{N}-b_{N})/a_{N}$ for $N\in\{10^{2},10^{4},10^{6}\}$, which approaches the density of $X$ (namely, $f_{X}(x)=\exp(x-e^{x})$) as $N$ increases. In both Figures~\ref{figmean} and \ref{figdist}, we take $L=D=1$.

\begin{figure*}
\centering
\includegraphics[width=.99\linewidth]{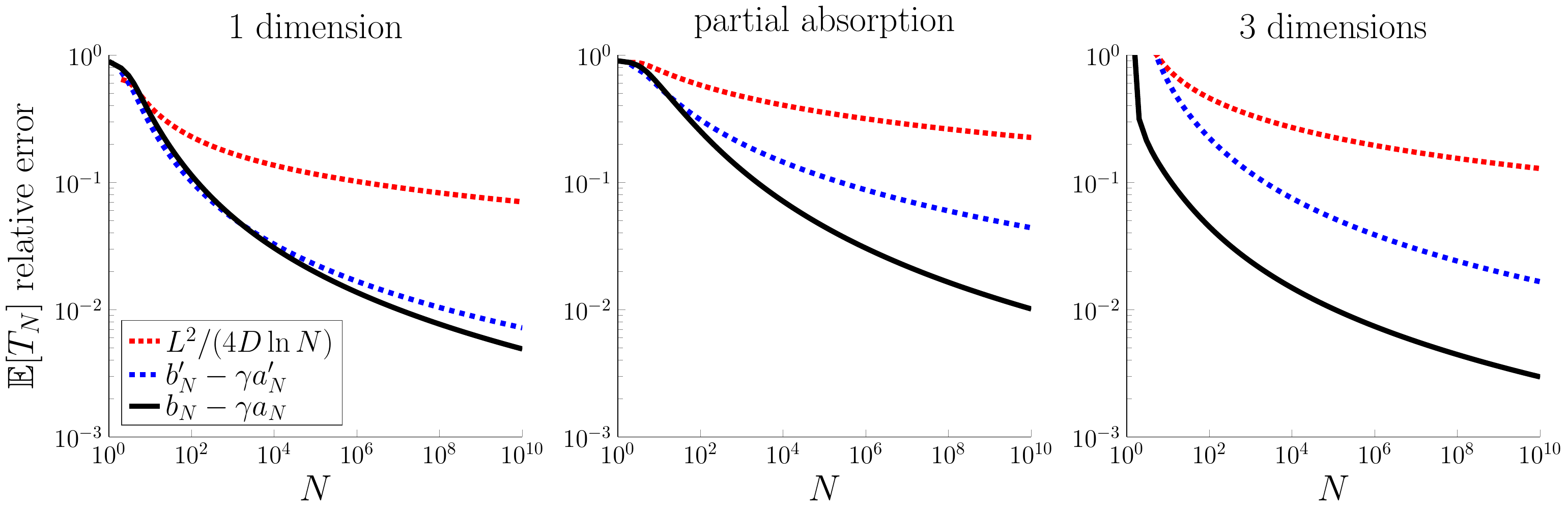}
\caption{Accuracy of approximations to the mean fastest FPT. The left, middle, and right panels correspond respectively to the examples in sections~\ref{1d}, \ref{partial}, and \ref{3d}. In each panel, the red dotted curve is the relative error \eqref{re} for the approximation $L^{2}/(4D\ln N)$, the blue dashed curve is for $b_{N}'-\gamma a_{N}'$ where $a_{N}',b_{N}'$ are given by \eqref{ababuf}, and the black solid curve is for $b_{N}-\gamma a_{N}$ where $a_{N},b_{N}$ are given by \eqref{abab}.}
\label{figmean}
\end{figure*}

\begin{figure*}
\centering
\includegraphics[width=.99\linewidth]{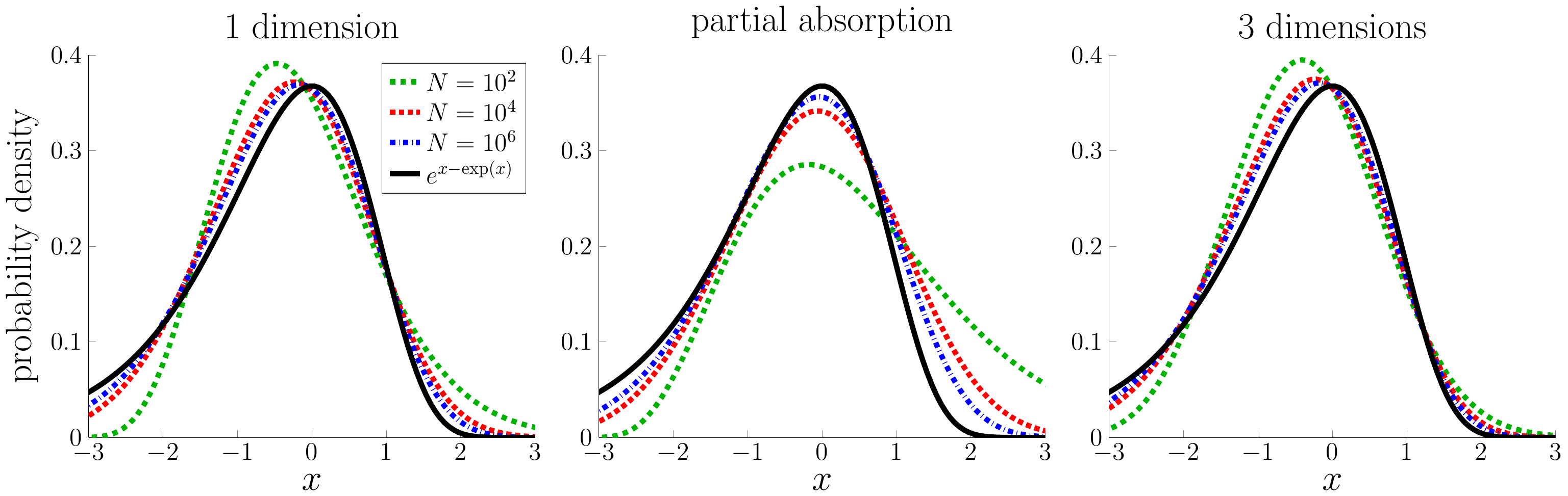}
\caption{Convergence in distribution of rescaled fastest FPT. The left, middle, and right panels correspond respectively to the examples in sections~\ref{1d}, \ref{partial}, and \ref{3d}. In each panel, the green dashed, red dotted, and blue dashed-dotted curves give the probability density function of $(T_{N}-b_{N})/a_{N}$ for $N\in\{10^{2},10^{4},10^{6}\}$ where $\{a_{N}\}_{N\ge1}$ and $\{b_{N}\}_{N\ge1}$ are given by \eqref{abab}, and the black solid curve shows the limiting probability density function, $\exp(x-e^{x})$.}
\label{figdist}
\end{figure*}

\subsection{Partial absorption}\label{partial}

Consider the example in the previous subsection, but now suppose the target at the origin is partially absorbing. This means that when a searcher hits the target, it is either absorbed or reflected, and the probabilities of these two events are described by a parameter $\kappa>0$ called the reactivity or absorption rate \citep{grebenkov2006}. Mathematically, this means the Fokker-Planck equation describing the probability density for a searcher's position has a Robin boundary condition at the origin involving the parameter $\kappa>0$ \citep{lawley15jk}.

In this case, the survival probability for a single searcher is  {\black \citep{carslaw1959}}
\begin{align*}
S(t)
=\P(\tau_{1}>t)
=1-\text{erfc}\Big(\frac{{{L}}}{\sqrt{4D t}}\Big)+e^{\frac{\kappa  (\kappa  t+{{L}})}{D}} \text{erfc}\Big(\frac{2 \kappa  t+{{L}}}{\sqrt{4D t}}\Big),\quad t>0,
\end{align*}
and thus
\begin{align*}
1-S(t)
\sim \frac{4}{\sqrt{\pi}}\frac{\kappa {L}}{D}\Big(\frac{Dt}{{L}^{2}}\Big)^{3/2}e^{-{L}^{2}/(4Dt)}\quad\text{as }t\to0+.
\end{align*}
Therefore, Theorems~\ref{main}-\ref{kth moment} hold with
\begin{align}\label{acppartial}
A=\frac{4}{\sqrt{\pi}}\frac{\kappa {L}}{D}\Big(\frac{D}{{L}^{2}}\Big)^{3/2},\quad
p=\frac{3}{2},\quad
C=\frac{{L}^{2}}{4D}.
\end{align}

The middle panel of Figure~\ref{figmean} shows the relative error \eqref{re} of approximations to $\E[T_{N}]$ in this case of a partially absorbing target. The red dotted curve is again the error for the approximation $\T_{N}=L^{2}/(4D\ln N)$ (this approximation was recently found and proven to have the correct large $N$ asymptotics \citep{lawley2020uni}). The blue dashed and black solid curves again correspond respectively to $\T_{N}=b_{N}'-\gamma a_{N}'$ and $\T_{N}=b_{N}-\gamma a_{N}$ where $a_{N}',b_{N}'$ are in \eqref{ababuf} and $a_{N},b_{N}$ are in \eqref{abab}. Again, $b_{N}-\gamma a_{N}$ and $b_{N}'-\gamma a_{N}'$ are much more accurate than $L^{2}/(4D\ln N)$, and $b_{N}-\gamma a_{N}$ is more accurate than $b_{N}'-\gamma a_{N}'$.

The middle panel of Figure~\ref{figdist} illustrates the convergence in distribution of $(T_{N}-b_{N})/a_{N}$ to $X=_{\textup{d}}\textup{Gumbel}(0,1)$ in this case of a partially absorbing target (again, for $N\in\{10^{2},10^{4},10^{6}\}$ and where $\{a_{N}\}_{N\ge1}$ and $\{b_{N}\}_{N\ge1}$ are given by \eqref{abab}). In both Figures~\ref{figmean} and \ref{figdist}, we take $L=D=\kappa=1$.

\subsection{Three dimensions}\label{3d}

Finally, consider the case where the $N\ge1$ independent searchers diffuse in three dimensional space, and let $\tau_{n}$ be the first time the $n$th searcher leaves a sphere of radius $L>0$ centered at its starting location. In this case, the survival probability for a single searcher is  {\black \citep{carslaw1959}}
\begin{align*}
S(t)
=\P(\tau_{1}>t)
=1-2\sqrt{\frac{L^{2}}{\pi Dt}}\sum_{j=0}^{\infty}e^{-(j+1/2)^{2}L^{2}/(Dt)},\quad t>0,
\end{align*}
and thus
\begin{align*}
1-S(t)
\sim 2\sqrt{\frac{L^{2}}{\pi Dt}}e^{-L^{2}/(4Dt)}\quad\text{as }t\to0+.
\end{align*}
Therefore, Theorems~\ref{main}-\ref{kth moment} hold with
\begin{align*}
A=2\sqrt{\frac{L^{2}}{\pi D}},\quad
p=-\frac{1}{2},\quad
C=\frac{{L}^{2}}{4D}.
\end{align*}

The right panel of Figure~\ref{figmean} shows the relative error \eqref{re} of approximations to $\E[T_{N}]$ in this three dimensional example. The red dotted curve is again the error for the approximation $\T_{N}=L^{2}/(4D\ln N)$ (this approximation was found by \cite{yuste2001}). The blue dashed and black solid curves again correspond respectively to $\T_{N}=b_{N}'-\gamma a_{N}'$ and $\T_{N}=b_{N}-\gamma a_{N}$ where $a_{N}',b_{N}'$ are in \eqref{ababuf} and $a_{N},b_{N}$ are in \eqref{abab}. Further, the right panel of Figure~\ref{figdist} illustrates the convergence in distribution of $(T_{N}-b_{N})/a_{N}$ to $X=_{\textup{d}}\textup{Gumbel}(0,1)$ in this three dimensional example (again, for $N\in\{10^{2},10^{4},10^{6}\}$ and where $\{a_{N}\}_{N\ge1}$ and $\{b_{N}\}_{N\ge1}$ are given by \eqref{abab}). In both Figures~\ref{figmean} and \ref{figdist}, we take $L=D=1$.

\section{Discussion}

In this work, we found tractable approximations for the full probability distribution of extreme FPTs of diffusion. These approximate distributions depend on only three parameters describing the short time behavior of the survival probability of a single searcher, and we proved that these approximations are exact in the many searcher limit. We used our approximate distributions to derive new formulas for statistics of extreme FPTs and prove rigorous error estimates.

Extreme FPTs of diffusion were first studied by \citet*{weiss1983}, where they found approximations of $\E[T_{k,N}]$ for large $N$ in various one dimensional problems. Statistics of extreme FPTs of diffusion in one dimensional or spherically symmetric domains were further studied by \cite{yuste1996}, \cite{yuste2000}, \cite{yuste2001}, \cite{van2003}, \cite{redner2014}, and \cite{meerson2015}. Recently, approximate formulas for the moments of extreme FPTs of diffusion in more general two and three dimensional domains were derived by \cite{ro2017}, \cite{basnayake2019}, and \cite{lawley2020esp}. Even more recently, it was proven in significant generality that the $m$th moment of the $k$th fastest FPT has the large $N$ behavior,
\begin{align}\label{universal}
\E[(T_{k,N})^{m}]
\sim\Big(\frac{L^{2}}{4D\ln N}\Big)^{m}\quad\text{as }N\to\infty,
\end{align}
where $D$ is a characteristic diffusivity and $L$ is a certain geodesic distance \citep{lawley2020uni}.

The moment formulas derived in the present work agree with \eqref{universal} to leading order, but are much more accurate for finite $N$. In addition, the moment formulas in the present work explain and confirm a remarkable conjecture by \cite{yuste2000}. In that work, the authors conjectured that the mean fastest FPT to escape a ball of radius $L$ in dimension $d\ge2$ has the following approximation,
\begin{align}\label{yconjecture}
\E[T_{1,N}]
\approx\frac{L^{2}}{4D\ln N}\Big[1+\sum_{n=1}^{\infty}(\ln N)^{-n}\sum_{m=0}^{n}K_{m}^{(n)}(\ln\ln N)^{m}\Big],
\end{align}
for some unknown constants $\{\{K_{m}^{(n)}\}_{m=0}^{n}\}_{n\ge1}$ (some of which were estimated numerically). To derive \eqref{yconjecture} from our results, first note that the principal branch of the LambertW function has the following expansion for $z\gg1$ \citep{corless1996},
\begin{align}\label{wexpand}
\begin{split}
W_{0}(z)
&= L_{1} - L_{2} + \sum_{i=0}^\infty \sum_{j=1}^\infty c_{ij} L_1^{-i-j} L_{2}^j,
\end{split}
\end{align}
where $L_{1}=\ln z$, $L_{2}=\ln \ln z$, 
\begin{align*}
c_{ij}
=\frac{(-1)^i}{j!} \left[ \begin{matrix} i + j \\ i + 1 \end{matrix} \right],
\end{align*}
 and $\left[ \begin{smallmatrix} i + j \\ i + 1 \end{smallmatrix} \right]$ are non-negative Stirling numbers of the first kind. Similarly, the lower branch, $W_{-1}(z)$, has the expansion in \eqref{wexpand} for $-1\ll z<0$ if $L_{1}=\ln(-z)$ and $L_{2}=\ln(-\ln(-z))$ \citep{corless1996}. Therefore, upon using the definitions in \eqref{abab}-\eqref{dubdub} and the expansion in \eqref{wexpand}, it follows that our formula $\E[T_{1,N}]\approx b_{N}-\gamma a_{N}$ is exactly of the form in the conjecture \eqref{yconjecture}.

Finally, we emphasize that our results apply to any FPT problem where the survival probability $S(t)=\P(\tau_{1}>t)$ of a single searcher satisfies
\begin{align}\label{st2}
1-S(t)\sim At^{p}e^{-C/t}\quad\text{as }t\to0+,
\end{align}
for some constants $C>0$, $A>0$, and $p\in\R$. The behavior in \eqref{st2} is very generic for diffusion processes and holds in many diverse scenarios. For example, \cite{weiss1983} found \eqref{st2} for one-dimensional drift-diffusion processes with a broad class of potential (drift) fields. Similarly, \cite{yuste2001} found \eqref{st2} for the first time a pure diffusion in dimension $d\ge1$ moves any distance $L>0$ from its initial location (and referred to \eqref{st2} as a ``universal'' form). Further, \cite{ro2017} formally derived \eqref{st2} for a pure diffusion searching for an arbitrarily placed small target in a spherical domain in dimension $d=3$. It is also known that \eqref{st2} holds for pure diffusion in dimension $d=1$ with a partially absorbing target (see section~\ref{partial} above). Further, it was proven that under very general conditions (including (i) diffusions in $\R^{d}$ with space-dependent diffusivities and drift fields and (ii) diffusions on $d$-dimensional smooth Riemannian manifolds that may contain reflecting obstacles), the survival probability satisfies \citep{lawley2020uni}
\begin{align}\label{vgg}
\lim_{t\to0+}t\ln(1-S(t))
=-\frac{L^{2}}{4D}<0,
\end{align}
where $D>0$ is a characteristic diffusivity and $L>0$ is a certain geodesic distance that depends on any space-dependence or anisotropy in the diffusivity (if the diffusivity is constant in space, then $L$ is merely the shortest distance from the starting location to the target). Therefore, if \eqref{st2} holds in a particular problem, then \eqref{vgg} implies that $C=L^{2}/(4D)$, and thus the only parameters to be found are $A$ and $p$.

{\black Finally, we discuss how our results can be used} to investigate the question posed in the Introduction section regarding the competing limits of small targets and many searchers, which is a ubiquitous feature of extreme FPTs in biological applications \citep{schuss2019}. {\black Consider a collection of small targets on an otherwise reflecting surface. Such a ``patchy surface'' arises in many applications \citep{brown1900,wolf2016,keil1999}, including the classic work of \cite{berg1977} on cell sensing by membrane receptors. In this scenario, the heterogeneous surface is commonly replaced by a homogeneous partially absorbing surface with a reactivity parameter $\kappa>0$ which incorporates the size, number, and arrangement of targets (many methods have been developed for determining $\kappa$ \citep{Berezhkovskii2004,Muratov2008,cheviakov2012,Dagdug2016,bernoff2018b,lawley2019fpk}). We thus cast this scenario into the setup of Section~\ref{partial} above. If the targets are small, then the trapping rate is small, $\kappa L/D\ll1$ \citep{bernoff2018b}. Hence, the mean fastest FPT will be large compared to the diffusion timescale unless the number of searchers is very large.} How large does the number of searchers $N$ need to be in order for the mean fastest FPT to be small compared to the diffusion timescale? That is, (i) how large does $N$ need to be so that $\frac{D}{L^{2}}\E[T_{N}]\ll1$ and (ii) what is an approximation of $\E[T_{N}]$ in this regime?

To answer questions (i) and (ii), {\black  notice that Theorems~\ref{uf} and \ref{moments} imply that 
\begin{align*}
\E[T_{N}]
&=b_{N}-\gamma a_{N}+o(a_{N})\\
&=\frac{C}{\ln N}\Big[1+\frac{p\ln(\ln(N))}{\ln N}-\frac{\ln(AC^{p})+\gamma}{\ln N}+o(1/\ln N)\Big]\quad\text{as }N\to\infty,
\end{align*}
where are $A$, $p$, and $C$ are given in \eqref{acppartial}.} Hence, this approximation is accurate if $N$ is sufficiently large so that
\begin{align}\label{ineq}
|\ln(AC^{p})|
\ll\ln N.
\end{align}
Using the values in {\black \eqref{acppartial}}, the relation \eqref{ineq} becomes {\black $\ln\tfrac{D}{\kappa L}\ll\ln N$}.
Therefore, in this scenario we have that 
\begin{align*}
\E[T_{N}]
\approx b_{N}-\gamma a_{N}
\ll \tfrac{L^{2}}{D}\quad\text{as long as }\textcolor{black}{\ln\tfrac{D}{\kappa L}\ll\ln N},
\end{align*}
which answers questions (i) and (ii) above.

\section{Appendix}

In this appendix, we collect the proofs of the propositions and theorems of section~\ref{math}.

\begin{proof}[Proof of Proposition~\ref{basic}]
This proposition merely collects basic results on Gumbel random variables, all of which follow directly from \eqref{xgumbel}.
\end{proof}

\begin{proof}[Proof of Proposition~\ref{easy}]
Since most results in extreme value theory are formulated in terms of the maximum of a set of random variables, define
\begin{align*}
M_{N}:=\max\{-\tau_{1},\dots,-\tau_{N}\}=-T_{N},
\end{align*}
and $F(x)=\P(-\tau_{1}<x)=S(-x)$. If there exists normalizing constants $\{a_{N}\}_{N\ge1}$ and $\{b_{N}\}_{N\ge1}$ so that $(M_{N}-b_{N})/a_{N}$ converges in distribution as $N\to\infty$ to a nontrivial random variable, then the distribution of that random variable can only be Frechet, Weibull, or Gumbel \citep{fisher1928}. Since 
\begin{align*}
x^{*}
:=\sup\{x:F(x)<1\}=0<\infty,
\end{align*}
Theorem 1.2.1 in the book by \cite{haanbook} ensures that the limiting distribution cannot be Frechet.

Furthermore, if the limiting distribution is Weibull, then Theorem 1.2.1 in \citep{haanbook} guarantees that there exists some $\gamma<0$ so that
\begin{align}\label{must}
\lim_{t\to0+}\frac{1-F(-tx)}{1-F(-t)}
=\lim_{t\to0+}\frac{1-S(tx)}{1-S(t)}
=x^{-1/\gamma}\quad\text{for all }x>0.
\end{align}
Now, it follows directly from \eqref{C} that $S(t)=1-e^{-C/t+h(t)}$ for some function $h(t)$ satisfying
\begin{align}\label{h}
\lim_{t\to0+}th(t)=0.
\end{align}

Therefore, we claim that \eqref{must} is violated with, for example, $x=2$. To see this, note that
\begin{align*}
\lim_{t\to0+}\frac{1-S(2t)}{1-S(t)}
=\lim_{t\to0+}e^{C/t+h(2t)-h(t)}.
\end{align*}
By \eqref{h}, we are assured that
\begin{align*}
-\frac{C}{2t}
\le h(t)
\le\frac{C}{2t}\quad\text{for sufficiently small }t.
\end{align*}
Hence,
\begin{align*}
\lim_{t\to0+}e^{C/t+h(2t)-h(t)}
\ge\lim_{t\to0+}e^{3C/(4t)}=+\infty,
\end{align*}
which indeed violates \eqref{must}. Therefore, if the limiting distribution is nondegenerate, it must be Gumbel.
\end{proof}

\begin{proof}[Proof of Proposition~\ref{hh}]
Using the assumptions on $h(t)$ in \eqref{hc}, a direct calculation shows that
\begin{align*}
\lim_{t\to0+}\frac{\dd}{\dd t}\left(\frac{1-S_{0}(t)}{S_{0}'(t)}\right)
=0.
\end{align*}
Therefore, Theorem 2.1.2 in \citep{falkbook} ensures that
\begin{align}\label{el}
\lim_{N\to\infty}(S_{0}(a_{N}x+b_{N}))^{N}
=\exp(-e^{x}),\quad\text{for all }x\in\R,
\end{align}
for some choice of normalizing constants $\{(a_{N},b_{N})\}_{N\ge1}$. Remark 1.1.9 in the book by \cite{haanbook} implies we can take $a_{N}$ and $b_{N}$ as in \eqref{ab}.

Now, \eqref{el} is equivalent to
\begin{align*}
\lim_{N\to\infty}N
\ln(S_{0}(a_{N}x+b_{N}))
=-e^{x},\quad\text{for all }x\in\R.
\end{align*}
Hence, it must be the case that $S_{0}(a_{N}x+b_{N})\to1$ as $N\to\infty$, and thus a straightforward application of L'Hospital's rule gives
\begin{align*}
-\ln(S_{0}(a_{N}x+b_{N}))
\sim 1-S_{0}(a_{N}x+b_{N})\quad\text{as }N\to\infty.
\end{align*}
Therefore, \eqref{el} is equivalent to
\begin{align}\label{repl}
\lim_{N\to\infty}N(1-S_{0}(a_{N}x+b_{N}))
=e^{x},\quad\text{for all }x\in\R.
\end{align}
Now, $1-S_{0}(t)\sim1-S_{0}(t)$ as $t\to0+$ by assumption. Hence, \eqref{repl} holds with $S_{0}$ replaced by $S$, which then implies that \eqref{el} holds with $S_{0}$ replaced by $S$, which completes the proof.
\end{proof}

\begin{proof}[Proof of Theorem~\ref{main}]
The theorem follows from Proposition~\ref{hh} upon calculating $\{a_{N}\}_{N\ge1}$ and $\{b_{N}\}_{N\ge1}$ in \eqref{abab} for $h(t)=\ln(At^{p})$ and using properties of the LambertW function \citep{corless1996}.
\end{proof}

\begin{proof}[Proof of Theorem~\ref{uf}]
The theorem follows immediately from Theorem~\ref{main} and \eqref{cc0}-\eqref{cc1}.
\end{proof}

\begin{proof}[Proof of Theorem~\ref{moments}]
By assumption, $\E[T_{N}]<\infty$ for some $N\ge1$. Hence, if $m\in(0,1)$, then $\E[(T_{N})^{m}]\le1+\E[T_{N}]<\infty$. If $m\ge1$, then it is straightforward to check that (see the proof of Proposition 2 in the work by \cite{lawley2020uni})
\begin{align*}
\E[(T_{2^{m-1}N})^{m}]<\infty.
\end{align*}
Since $\E[X^{m}]<\infty$, applying Theorem 2.1 in \citep{pickands1968} completes the proof.
\end{proof}

\begin{proof}[Proof of Theorem~\ref{kth}]
The convergence in distribution in \eqref{cd1} and \eqref{cd2} follows immediately from Theorem~\ref{main} above and Theorem 3.5 in the book by \cite{colesbook}.
\end{proof}

\begin{proof}[Proof of Theorem~\ref{kth moment}]
While convergence in distribution does not necessarily imply convergence of moments, it does imply convergence of moments if the sequence of random variables is uniformly integrable \citep{billingsley2013}. Hence, it is sufficient to prove that
\begin{align}\label{suffc}
\sup_{N}\E\Big[\Big(\frac{T_{k,N}-b_{N}}{a_{N}}\Big)^{2}\Big]
<\infty
\end{align}
since \eqref{suffc} ensures that $\{\frac{T_{k,N}-b_{N}}{a_{N}}\}_{N\ge1}$ is uniformly integrable \citep{billingsley2013}.

By assumption, $1-S(t)\sim At^{p}e^{-C/t}$ as $t\to0+$. Hence, there exists a $\delta>0$ so that
\begin{align*}
1-A_{1}t^{p}e^{-C/t}
\le 1-S(t)
\le 1-A_{0}t^{p}e^{-C/t},\quad\text{if }t\in(0,\delta],
\end{align*}
where $0<A_{0}<A<A_{1}$. Define the survival probability
\begin{align*}
S_{+}(t)
=\begin{cases}
1 & t\le0,\\
1-A_{0}t^{p}e^{-C/t} & t\in(0,\delta],\\
S(t) & t>\delta.
\end{cases}
\end{align*}
Define $S_{-}(t)$ similarly with $A_{0}$ replaced by $A_{1}$. Hence, $S_{-}(t)\le S(t)\le S_{+}(t)$ for all $t\in\R$. Let $\{U_{n}\}_{n\ge1}$ be an iid sequence of random variables, each with a uniform distribution on $[0,1]$. Define
\begin{align*}
\tau_{n}
&:=S^{-1}(U_{n}),\\
\tau_{n}^{-}
&:=S_{-}^{-1}(U_{n}),\\
\tau_{n}^{+}
&:=S_{+}^{-1}(U_{n}),
\end{align*}
and
\begin{align*}
T_{k,N}
&:=\min\big\{\{\tau_{1},\dots,\tau_{N}\}\backslash\cup_{j=1}^{k-1}\{T_{j,N}\}\big\},\quad k\in\{1,\dots,N\},\\
T_{k,N}^{\pm}
&:=\min\big\{\{\tau_{1}^{\pm},\dots,\tau_{N}^{\pm}\}\backslash\cup_{j=1}^{k-1}\{T_{j,N}^{\pm}\}\big\},\quad k\in\{1,\dots,N\},
\end{align*}
where $T_{1,N}:=\min\{\tau_{1},\dots,\tau_{N}\}$ and $T_{1,N}^{\pm}:=\min\{\tau_{1}^{\pm},\dots,\tau_{N}^{\pm}\}$. By construction, we have that
\begin{align*}
T_{k,N}^{-}
\le T_{k,N}
\le T_{k,N}^{+}\quad\text{almost surely}.
\end{align*}

Therefore, if $1_{\mathcal{A}}$ denotes the indicator function on an event $\mathcal{A}$, then
\begin{align*}
\E\Big[\Big(\frac{T_{k,N}-b_{N}}{a_{N}}\Big)^{2}\Big]
&=\E\Big[\Big(\frac{T_{k,N}-b_{N}}{a_{N}}\Big)^{2}1_{T_{k,N}>b_{N}}\Big]
+\E\Big[\Big(\frac{T_{k,N}-b_{N}}{a_{N}}\Big)^{2}1_{T_{k,N}<b_{N}}\Big]\\
&\le\E\Big[\Big(\frac{T_{k,N}^{+}-b_{N}}{a_{N}}\Big)^{2}1_{T_{k,N}^{+}>b_{N}}\Big]
+\E\Big[\Big(\frac{T_{k,N}^{-}-b_{N}}{a_{N}}\Big)^{2}1_{T_{k,N}^{-}<b_{N}}\Big]\\
&\le\E\Big[\Big(\frac{T_{k,N}^{+}-b_{N}}{a_{N}}\Big)^{2}\Big]
+\E\Big[\Big(\frac{T_{k,N}^{-}-b_{N}}{a_{N}}\Big)^{2}\Big].
\end{align*}
Hence, it remains to show that
\begin{align*}
\sup_{N}\E\Big[\Big(\frac{T_{k,N}^{\pm}-b_{N}}{a_{N}}\Big)^{2}\Big]
<\infty.
\end{align*}

Now,
\begin{align*}
\E\Big[\Big(\frac{T_{k,N}^{\pm}-b_{N}}{a_{N}}\Big)^{2}\Big]
&=\int_{0}^{\infty}\P\Big(\Big(\frac{T_{k,N}^{\pm}-b_{N}}{a_{N}}\Big)^{2}>t\Big)\,\dd t\\
&=\int_{0}^{\infty}\P(T_{k,N}^{\pm}-b_{N}>a_{N}\sqrt{t})\,\dd t
+\int_{0}^{\infty}\P(b_{N}-T_{k,N}^{\pm}>a_{N}\sqrt{t})\,\dd t\\
&=:I_{1}+I_{2}.
\end{align*}
Since $T_{1,N}^{\pm}\le T_{k,N}^{\pm}$ almost surely for any $k\in\{1,\dots,n\}$, we have that
\begin{align*}
I_{2}
\le\int_{0}^{\infty}\P(b_{N}-T_{1,N}^{\pm}>a_{N}\sqrt{t})\,\dd t
&\le\int_{0}^{\infty}\P\Big(\Big(\frac{T_{1,N}^{\pm}-b_{N}}{a_{N}}\Big)^{2}>t\Big)\,\dd t\\
&=\E\Big[\Big(\frac{T_{1,N}^{\pm}-b_{N}}{a_{N}}\Big)^{2}\Big].
\end{align*}
Now, Theorem~\ref{moments} implies that
\begin{align*}
\E\Big[\Big(\frac{T_{1,N}^{\pm}-b_{N}^{\pm}}{a_{N}^{\pm}}\Big)^{2}\Big]
\to\E[X^{2}]<\infty\quad\text{as }N\to\infty,
\end{align*}
where $\{a_{N}^{\pm}\}_{N\ge1}$ and $\{b_{N}^{\pm}\}_{N\ge1}$ are given by \eqref{abab} with $A$ replaced by $A_{0}$ or $A_{1}$. Now, it is straightforward to check that there exists $\alpha^{\pm}>0$ and $\beta^{\pm}\in\R$ so that
\begin{align*}
\frac{a_{N}^{\pm}}{a_{N}}\to\alpha^{\pm}
\quad\text{and}\quad
\frac{b_{N}-b_{N}^{\pm}}{a_{N}}\to\beta^{\pm}\quad\text{as }N\to\infty.
\end{align*}
Therefore, Proposition 1.1 and Remark 1 in the work by \cite{peng2012} imply that
$\E[(\frac{T_{1,N}^{\pm}-b_{N}^{\pm}}{a_{N}^{\pm}})^{2}]$ converges to some finite constant as $N\to\infty$. Hence,
\begin{align*}
\sup_{N}I_{2}<\infty.
\end{align*}

Moving to $I_{1}$, note first that
\begin{align*}
\P(T_{k,N}^{\pm}>x)
=\P(T_{1,N}^{\pm}>x)
+\sum_{j=1}^{k-1}\P(T_{j,N}^{\pm}<x<T_{j+1,N}^{\pm}).
\end{align*}
Hence,
\begin{align*}
I_{1}
&=\int_{0}^{\infty}\P(T_{1,N}^{\pm}>a_{N}\sqrt{t}+b_{N})\,\dd t
+\sum_{j=1}^{k-1}\int_{0}^{\infty}\P(T_{j,N}^{\pm}<a_{N}\sqrt{t}+b_{N}<T_{j+1,N}^{\pm})\,\dd t\\
&=:I_{3}+I_{4}.
\end{align*}
Now, $I_{3}$ can be handled similarly to $I_{2}$ to obtain
\begin{align*}
\sup_{N}I_{3}<\infty.
\end{align*}

Hence, it remains to show that $\sup_{N}I_{4}<\infty$. Now, since $\{\tau_{n}^{\pm}\}_{n\ge1}$ are iid, it follows that
\begin{align*}
\P(T_{j,N}^{\pm}<x<T_{j+1,N}^{\pm})
={N \choose j}(1-S_{\pm}(x))^{j}(S_{\pm}(x))^{N-j},\;\text{if }j\in\{1,\dots,k-1\}.
\end{align*}
Hence,
\begin{align*}
I_{4}
=\sum_{j=1}^{k-1}{N \choose j}\int_{0}^{\infty}(1-S_{\pm}(a_{N}\sqrt{t}+b_{N}))^{j}(S_{\pm}(a_{N}\sqrt{t}+b_{N}))^{N-j}\,\dd t.
\end{align*}
An application of Laplace's method with a movable maximum (see, for example, section 6.4 in the book by \cite{bender2013}) shows that each term in this sum is bounded in $N$, and so the proof is complete.
\end{proof}

%
%

\bibliographystyle{spbasic}
\bibliography{library}

\end{document}